\documentclass{amsart}
\author{Johannes Nicaise, D. Peter Overholser, Helge Ruddat}
\usepackage{graphicx}
\usepackage{caption}
\usepackage{subcaption}
\title{Motivic zeta functions of the quartic and its mirror dual}

\usepackage{amssymb}
\usepackage{tikz}
\usetikzlibrary{arrows}
\usetikzlibrary{decorations.markings}
\tikzstyle arrowstyle=[scale=1]
\usepackage{pgfplots}
\pgfplotsset{compat=newest}
\usetikzlibrary{calc}
\tikzstyle dir=[postaction={decorate,decoration={markings,
    mark=at position .65 with {\arrow[arrowstyle]{stealth}}}}]
\tikzstyle rdir=[postaction={decorate,decoration={markings,
    mark=at position .65 with {\arrowreversed[arrowstyle]{stealth};}}}]
\usepackage{graphicx}
\usepackage{amsthm}
\usepackage{amsrefs}
\usepackage{amsmath}
\usepackage{subfig}
\usepackage{graphicx}
\usepackage{amsmath}
\usepackage{latexsym}
\usepackage{eepic}
\usepackage{epsfig}
\usepackage{graphicx}
\usepackage{amscd}
\usepackage{mathrsfs}
\usepackage[english]{babel}

\newcommand {\Aut} {\operatorname{Aut}}
\newcommand {\ZZ} {\mathbb{Z}}
\newcommand {\QQ} {\mathbb{Q}}
\newcommand {\RR} {\mathbb{R}}
\newcommand {\CC} {\mathbb{C}}
\newcommand {\PP} {\mathbb{P}}
\renewcommand {\P} {\mathcal{P}}
\newcommand {\ra} {\rightarrow}
\newcommand {\Spec} {\operatorname{Spec}}

\newcommand{\A}{\mathbb{A}}
\newcommand{\N}{\mathbb{N}}
\newcommand{\Z}{\mathbb{Z}}
\newcommand{\C}{\mathbb{C}}
\newcommand{\Q}{\mathbb{Q}}
\newcommand{\LL}{\mathbb{L}}
\newcommand{\Var}{\mathrm{Var}}

\newcommand{\llbr}{[\negthinspace[}
\newcommand{\rrbr}{]\negthinspace]}

\newcommand{\llpar}{(\negthinspace(}
\newcommand{\rrpar}{)\negthinspace)}

\newtheorem{theorem}{Theorem}[section]

\newtheorem{lemma}[theorem]{Lemma}

\newtheorem{proposition}[theorem]{Proposition}
\newtheorem{corollary}[theorem]{Corollary}

\theoremstyle{definition}

\newtheorem{definition-lemma}[theorem]{Definition-Lemma}

\theoremstyle{remark}

\numberwithin{equation}{section}
\numberwithin{figure}{section}
\theoremstyle{definition}

\newcounter{countera}
\newcounter{counterb}

\begin{document}
\begin{abstract}
We use a formula of Bultot to compute the motivic zeta function for the toric degeneration of the quartic K3 and its Gross-Siebert mirror dual degeneration. 
We check for this explicit example that the identification of the logarithm of the monodromy and the mirror dual Lefschetz operator works at an integral level.
\end{abstract}
\maketitle
\section{Introduction}
 Motivic integration was introduced by Kontsevich in a famous lecture at Orsay in 1995. Kontsevich invented this theory in order to prove that
  birationally equivalent complex Calabi-Yau varieties have the same Hodge numbers.

 Motivic integrals take values in a suitable Grothendieck ring of varieties and can be specialized to various additive invariants of algebraic varieties, such as Hodge-Deligne polynomials.

 An early application of motivic integration was Denef and Loeser's definition of the {\em motivic zeta function} of a hypersurface singularity, designed as a motivic upgrade of Igusa's $p$-adic local zeta function. This motivic zeta function is a very rich invariant of the singularity and contains several interesting classical invariants, such as the Hodge spectrum.

 We refer to \cite{Ni} for an elementary introduction. In \cite{HaNi2}, an analogous motivic zeta function was defined for a smooth and proper variety $X$ over $\C\llpar t\rrpar$ with trivial canonical sheaf.
  The case of abelian varieties was studied in detail in \cite{HaNi}, but only very few examples are known when $X$ is Calabi-Yau.

 A large class of Calabi-Yau varieties over $\C\llpar t\rrpar$ can be constructed and described by means of the theory of {\em toric degenerations} of Gross and Siebert. Using tropical and logarithmic geometry, Gross and Siebert showed how to construct a Calabi-Yau variety over $\C\llpar t\rrpar$ from combinatorial data: a topological manifold with an integral affine structure with singularities and a piecewise linear function on this manifold \cite{GrSi}. This construction explains mirror symmetry between Calabi-Yau varieties over $\C\llpar t\rrpar$ as a duality between the combinatorial data.

   We aim to use this construction to compute motivic zeta functions.  As a first step, and to illustrate the techniques that come into play, we work out the details in the case of a mirror pair of $K3$ surfaces, and we compare the result with the formula of Stewart and Vologodsky \cite{StVo}.
   \if false
    As shown in \cite{FS86}, the monodromy of a maximally unipotent monodromy degeneration of $K3$ surfaces is determined by a pair of integers, which we will denote by $(t,k)$.  Motivic zeta functions of semistable degenerations of $K3$-surfaces were previously studied by Stewart and Vologodsky \cite{StVo}, and they obtained a beautiful formula that expresses the motivic zeta function in terms of the limit mixed Hodge structure of the degeneration.  In the case of a maximally unipotent monodromy degeneration, it is completed determined by the integer $t$.  We reinterpret a formula of Bultot \cite{B15} for the motivic zeta function of a log smooth model in terms of the combinatorial data of the Gross-Siebert construction and naturally recover the result of Stewart and Vologodsky.  Finally, if we consider the action of the Lefschetz operator instead logarithm of the monodromy, we can extract an analogous pair of integers $(\check{t}, \check{k})$.  We show that mirror symmetry exchanges these pairs.
 \fi
\subsection*{Notation}
We will denote by $K$ the field of complex Laurent series $\C\llpar t\rrpar$ and by $R$ its valuation ring $\C\llbr t\rrbr$. If $\mathcal{X}$ is a scheme over $R$ then its special fiber will be denoted by $\mathcal{X}_0$. 
If $X$ is a separated $K$-scheme of finite type, then an $R$-model of $X$ is a flat separated $R$-scheme $\mathcal{X}$ of finite type, endowed with an isomorphism of $K$-schemes $\mathcal{X}_K\to X$ where $\mathcal{X}_K$ denotes the base change to $K$.

 We denote by $K_0(\Var_{\C})$ the Grothendieck ring of complex varieties, by $\LL=[\A^1_{\C}]\in K_0(\Var_{\C})$ the class of the affine line, and by $\mathcal{M}_{\C}=K_0(\Var_{\C})[\LL^{-1}]$ the localized Grothendieck ring. We recall that $K_0(\Var_{\C})$ is the quotient of the free abelian group on isomorphism classes $[X]$ of complex varieties $X$ modulo the scissor relations $[X]=[Y]+[X\setminus Y]$ whenever $Y$ is a closed subvariety of $X$. The product on $K_0(\Var_{\C})$ is induced by the fibered product of algebraic varieties over $\C$.

All the logarithmic structures in this paper are defined with respect to the Zariski topology.

\section{Motivic zeta functions of Calabi-Yau varieties and the formula of Stewart-Vologodsky}
\subsection{Motivic zeta functions of Calabi-Yau varieties}
 Let $X$ be a smooth, proper, geometrically connected variety over $K$ with trivial canonical line bundle, and assume that $X(K)$ is non-empty. The motivic zeta function
 $Z_{X}(T)$ is a formal power series in $T$ with coefficients in the localized Grothendieck ring $\mathcal{M}_{\C}$. It was introduced in \cite{HaNi2}, and measures the sets of rational points on $X$ over the finite extensions of $K$.
 In his PhD thesis, Bultot gave a formula for $Z_X(T)$ in terms of a  log-smooth $R$-model of $X$. We will now recall a particular case of his formula.

 Let $\mathcal{X}$ be a proper $R$-model of $X$. We endow $\Spec R$ with its standard log structure  and we endow $\mathcal{X}$ with the divisorial log structure induced by its special fiber $\mathcal{X}_0$. The sheaf of monoids defining the log structure on $\mathcal{X}$ will be denoted by $M_\mathcal{X}$.
 We assume that $\mathcal{X}$ is log-smooth over $R$, that $\mathcal{X}_0$ is reduced, and that the relative canonical line bundle $\omega^{\log}_{\mathcal{X}/R}$ (that is, the sheaf of relative log-differentials of maximal degree) is trivial. These assumptions substantially simplify Bultot's formula for the zeta function and they will be sufficient for the applications in this paper.

 We denote by $F(\mathcal{X})$ the {\em Kato fan} of $\mathcal{X}$. This is a topological space endowed with a sheaf of monoids.
  The underlying space of $F(\mathcal{X})$ is the subspace of $\mathcal{X}_0$ that consists of the generic points of intersections of sets of irreducible components of $\mathcal{X}_0$, and the monoid sheaf on $F(\mathcal{X})$ is the restriction of the sheaf of characteristic monoids $\overline{M}_{\mathcal{X}}=M_{\mathcal{X}}/\mathcal{O}_{\mathcal{X}}^{\times}$
 on $\mathcal{X}$.

  The special fiber $\mathcal{X}_0$ has a natural stratification, indexed by the points of the fan $F(\mathcal{X})$. For every $\xi$ in $F(\mathcal{X})$, we denote by
  $U(\xi)$ the subset of the closure of $\{\xi\}$ obtained by removing the closures of the sets $\{\xi'\}$ with $\xi'$ a point of $F(\mathcal{X})$ that does not specialize to $\xi$. The set $U(\xi)$ is a locally closed subset of $\mathcal{X}_0$, and we endow it with its reduced induced structure.

  For every point $\xi$ of $F(\mathcal{X})$, the stalk $\overline{M}_{\mathcal{X},\xi}$ is a toric monoid. The dimension of the monoid $\overline{M}_{\mathcal{X},\xi}$, which coincides with the rank of its groupification, will be denoted by $r(\xi)$.
 The monoid $\overline{M}_{\mathcal{X},\xi}$ contains a distinguished element $e_{\xi}$, which is the image of the residue class of the local parameter $t$ under
  the morphism
  $$\overline{M}_{\Spec R,0}=R/R^{\times}\cong \N\to \overline{M}_{\mathcal{X},\xi}$$ where we wrote $0$ for the closed point of $\Spec R$.
  We denote by $\overline{M}^{\vee,\mathrm{loc}}_{\mathcal{X},\xi}$ the set of local morphisms of monoids
  $$\phi:\overline{M}_{\mathcal{X},\xi}\to \N.$$ Here, ``local'' simply means that $\phi^{-1}(0)=\{1\}$.

 \begin{theorem}[Bultot]\label{thm:bultot}
With the above notations and assumptions, we have
$$Z_{X}(T)=\sum_{\xi\in F(\mathcal{X})}(\LL-1)^{r(\xi)-1}[U(\xi)]\sum_{\phi \in \overline{M}^{\vee,\mathrm{loc}}_{X\mathcal{X}\xi}}T^{\phi(e_{\xi})}\quad \in \mathcal{M}_{\C}[[T]].$$
 \end{theorem}
\begin{proof}
This is a particular case of \cite[3.1]{B15}.
\end{proof}

\subsection{The formula of Stewart-Vologodsky}
Here we will focus on the case where $X$ is a $K3$-surface with a strictly semistable $R$-model, i.e., a regular proper $R$-model $\mathcal{X}$  such that $\mathcal{X}_0$ is a reduced divisor with strict normal crossings. If we enlarge our category of models to include algebraic spaces over $R$ instead of only schemes, then we can find such a strictly semistable model $\mathcal{X}$ with the additional property that the relative canonical line bundle $\omega_{\mathcal{X}/R}$ is trivial. Such a model is called a {\em Kulikov model} for $X$. The possible special fibers of Kulikov models have been classified by Kulikov \cite{Kuli} and Persson-Pinkham \cite{PePi}. The shape of the special fiber is closely related to the limit mixed Hodge structure of $X$, and this allowed Stewart and Vologodsky to obtain an explicit formula for the motivic zeta function $Z_{X}(T)$ of $X$ in terms of the limit mixed Hodge structure \cite{StVo}.

\if false
As a first step, Stewart and Vologodsky defined cohomology spaces with $\Z$-coefficients $H^i(X\times_K K^a,\Z)$, where $K^a$ is an algebraic closure of $K$, and a mixed Hodge structure on these cohomology spaces (in \cite{StVo}, the notation $H^i(\lim X,\Z)$ is used instead). This construction is based on logarithmic geometry and the Kato-Nakayama space. Tensoring with $\Z_{\ell}$ (for any prime $\ell$) yields the usual $\ell$-adic cohomology spaces $H^i(X\times_K K^a,\Z_\ell)$.  If $X$ is obtained by base change from a proper family $X'\to C$ over a smooth complex curve $C$ with local coordinate $t$, then $H^i(X\times_K K^a,\Z)$ is canonically isomorphic to the nearby cohomology space $H^i(X'_0,R\psi_{X'}(\Z))$ where $X'_0$ denotes the fiber over $t=0$ and $\psi$ is the nearby cycles functor. In this case, the mixed Hodge structure on $H^i(X\times_K K^a,\Z)$ coincides with the limit mixed Hodge structure on $H^i(X'_0,R\psi_{X'}(\Z))$.

  The cohomology spaces $H^i(X\times_K K^a,\Z)$ are endowed with a unipotent endomorphism $\mathcal{T}$, called the monodromy transformation. The $K3$-surfaces $X$ are divided into types, according to the order of nilpotency of $\mathcal{T}-\mathrm{Id}$: we say that $X$ is of type $\tau$ if $\tau$ is the smallest positive integer such that $(\mathcal{T}-\mathrm{Id})^{\tau}=0$. 
  Toric degenerations of $K3$-surfaces in the sense of Gross and Siebert are always of type III (maximally unipotent monodromy).

\fi

  If $X$ is of type III and $\mathcal{X}$ is a Kulikov model, then the special fiber $\mathcal{X}_0$ is a union of smooth rational surfaces that intersect along cycles of smooth rational curves and the dual intersection complex of $\mathcal{X}_0$ is a triangulation of the $2$-sphere. Stewart and Vologodsky have shown that $Z_{X}(T)$ only depends on the number $t(X)$ of triple points in $\mathcal{X}_0$, i.e., the number of $2$-simplices in the dual intersection complex (this number is denoted by $r_2(X,K)$ in \cite{StVo}). The invariant $t(X)$ be read from the limit mixed Hodge structure on the degree $2$ cohomology of $X$ using \cite[7.1]{FS86}; see Section 4.
 
  \if false
  the following result of Friedman and Scattone \cite{FS86}.
   We consider the logarithm $N$ of the monodromy transformation $\mathcal{T}$:
   $$N=(\mathcal{T}-\mathrm{Id})-\frac{1}{2}(\mathcal{T}-\mathrm{Id})^2.$$ This is a nilpotent operator on
   $$H^2(X\times_K K^a,\Q)=H^2(X\times_K K^a,\Z)\otimes_{\Z}\Q$$ of order $3$, and one can show that the lattice
   $H^2(X\times_K K^a,\Z)$ is stable under $N$ (see \cite[1.2]{FS86}). We set
   $$W_j=W_jH^2(X\times_K K^a,\Q)\cap H^2(X\times_K K^a,\Z)$$ for every integer $j$.
   Then $N^2$ induces a morphism
  $$N^2:W_4/W_3\to W_0$$
   with finite cokernel, and the number $t(X)$ is equal to the order of this cokernel, by \cite[7.1]{FS86} and the proof of
   \cite[1.2]{FS86}. The proof of \cite[1.2]{FS86} also implies that $t(X)$ is always even. We can now formulate the formula for the zeta function as follows.
   \fi

   \begin{theorem}[Stewart-Vologodsky] \label{thm-StVo}
   If $X$ is a $K3$-surface over $K$ of type III, then
   $$Z_{X}(T)=\frac{t(X)}{2}(\LL-1)^2\frac{T(1+T)}{(1-T)^3}+(1+10\LL+\LL^2)\frac{2T}{1-T} \quad\in \mathcal{M}_{\C}[[T]].$$
   In particular, it only depends on $t(X)$, and it fully determines the invariant $t(X)$.
   \end{theorem}
   \begin{proof}
   This is an immediate consequence of Stewart and Vologodsky's formula (Theorem 1 in \cite{StVo}) for the coefficients of the generating series $Z_{X}(T)$.
   \if false
   , that is, the motivic integrals
   $$\int_{X(d)}|\omega(d)|$$ where $\omega$ is any normalized volume form on $X$.

   To see that $Z_{X}(T)$ determines $t(X)$, it suffices to note that
   the subring $\Z[\LL]$ of $\mathcal{M}_{\C}$ is isomorphic to the polynomial ring $\Z[u]$ {\em via} the ring morphism
   $$\mathcal{M}_{\C}\to \Z[u^{1/2},u^{-1/2}]$$ that sends a smooth and proper complex variety $Y$ to its Poincar\'e polynomial
   $$\sum_{i\geq 0}b_i(Y)u^{i/2}$$ where $b_i(Y)$ denotes the $i$-th Betti number of $Y$.
   \fi
   \end{proof}

 We can also use Theorem \ref{thm-StVo} to compute the {\em motivic volume} or {\em motivic nearby fiber} of $X$, which encodes the geometry of
 a general fiber of $X$ viewed as a degenerating family of complex varieties over a formal punctured disc. It is formally defined by taking the limit of $-Z_{X}(T)$ for
 $T\to +\infty$; see \cite[\S8]{NiSe}. This leads to the following result.

 \begin{corollary}
 If $X$ is a $K3$-surface over $K$ of type III, then the motivic volume of $X$ is equal to
 $$2+20\LL+2\LL^2 \quad \in \mathcal{M}_{\C}.$$
 \end{corollary}

 This expression nicely reflects the cohomological structure of $X$: subtracting the contributions $1$ and $\LL^2$ for the cohomology spaces of degree $0$ and $4$, respectively, we obtain the motivic decomposition $1+20\LL+\LL^2$ for the degree $2$ cohomology of $X$. The fact that the motivic volume lies in $\Z[\LL]$ reflects the property that the limit mixed Hodge structure is of Hodge-Tate type.

\section{The quartic and its mirror}
We begin by describing the affine manifolds which will serve as the dual intersection complexes of the degenerations of the quartic and its mirror.
\subsection{Affine manifolds and subdivisions}

The affine manifold (see \cite{GS03}) $B$ that is the intersection complex of a degeneration of the quartic can be constructed on the boundary of the Newton polytope of the quartic.
This Newton polytope $P$ is four times a standard $3$-simplex. Translating it so that its unique interior lattice point becomes the origin yields the
tetrahedron in $\mathbb{R}^3$ whose vertices are given by
\begin{equation}
\label{quarticpoly}
(-1,-1,-1), (3,-1,-1), (-1, 3, -1), \hbox{ and } (-1,-1,3).
\end{equation}
We define a subdivision $\mathcal{P}$ of $B:=\partial P$ by cutting each facet of the tetrahedron $P$ into 16 elementary two-simplices as shown in Figure \ref{fig:tetface}. We introduce a distinguished set $\Delta$ given by the barycenters of the 24 one-polytopes of $\mathcal{P}$ that are contained in an edge of the tetrahedron. These are marked in Figure \ref{fig:tetface} by dotted circles. The standard affine structure on the interior of each maximal simplex and a fan structure at each vertex, which is given as shown in Figure \ref{fig:fanstruct}, yield an integral affine structure on $B\setminus \Delta$ (a variation on example 2.10.4 of \cite{GS03}).  We equip $(B, \mathcal{P})$ with a piecewise linear function $\varphi$, defined by the Newton polytopes given in Figure \ref{fig:newton}.  Applying the discrete Legendre transform (see \cite[\S4]{GS03}), we obtain another integral tropical manifold $(\check{B}, \check{\mathcal{P}})$ with multi-valued piecewise linear function $\check\varphi$. A chart of $(\check{{B}}, \check{\mathcal{P}})$ is pictured in Figure \ref{fig:dualface}.
\begin{figure}
                 \resizebox{.5\textwidth}{!}{
                  \begin{tikzpicture}

    \coordinate (p0) at (0,0) ;
    \coordinate (p1) at (8,0) ;
    \coordinate (p2) at (0,8);
    \coordinate (p3) at (p0);
     \coordinate (p4) at (p1);
     \coordinate (offset) at (0.2,0.2);
 \foreach \x in {1,2,3}
    {       \setcounter{countera}{\x}
    	\setcounter{counterb}{\x}
	\addtocounter{countera}{1}
	\addtocounter{counterb}{-1}
	\foreach \n in {1,2,3,4}{
	\draw ($\n*.25*(p\thecountera)-\n*.25*(p\x)+(p\x)$) -- ($\n*.25*(p\thecounterb)-\n*.25*(p\x)+(p\x)$);
	\node at ($.25*\n*(p\thecountera)-.25*\n*(p\x)+(p\x)-.125*(p\thecountera)+.125*(p\x)$) {$\odot$};
	}
	\node at ($(p\x)+(offset)$) {${{a_1}}$};
	\node at ($(p\x)$) {$\bullet$};
	\node at ($(p\x)-(offset)$) {$10$};
	
	\node at ($.25*(p\thecountera)-.25*(p\x)+(p\x)+(offset)$) {${{a_2}}$};
	\node at ($.25*(p\thecountera)-.25*(p\x)+(p\x)$) {$\bullet$};
	\node at ($.25*(p\thecountera)-.25*(p\x)+(p\x)-(offset)$) {${7}$};

	\node at ($.75*(p\thecountera)-.75*(p\x)+(p\x)+(offset)$) {${{a_2}}$};
	\node at ($.75*(p\thecountera)-.75*(p\x)+(p\x)$) {$\bullet$};
	\node at ($.75*(p\thecountera)-.75*(p\x)+(p\x)-(offset)$) {${7}$};

	\node at ($.5*(p\thecountera)-.5*(p\x)+(p\x)+(offset)$) {${{a_3}}$};
	\node at ($.5*(p\thecountera)-.5*(p\x)+(p\x)$) {$\bullet$};
	\node at ($.5*(p\thecountera)-.5*(p\x)+(p\x)-(offset)$) {${6}$};
	
	\node at ($.25*(p\thecounterb)-.25*(p\x)+.5*(p\thecountera)-.5*(p\x)+(p\x)+(offset)$) {${{a_4}}$};
	\node at ($.25*(p\thecounterb)-.25*(p\x)+.5*(p\thecountera)-.5*(p\x)+(p\x)$) {$\bullet$};
	\node at ($.25*(p\thecounterb)-.25*(p\x)+.5*(p\thecountera)-.5*(p\x)+(p\x)-(offset)$) {${5}$};

	 }

 \end{tikzpicture}}
  \caption{The subdivision of a facet of the Newton polytope of the quartic. The union of facets gives the \emph{integral tropical manifold} $B$ with polyhedral decomposition $\mathcal{P}$. The pair $(B,\mathcal{P})$ is the intersection complex (``cone picture") of a degeneration ${\mathcal{X}}$ of a quartic K3 surface and the dual intersection complex (``fan picture") of a \emph{mirror} degeneration $\check{\mathcal{X}}$.  The symbol $\odot$ marks affine singularities, while the vertex labelings ${a_i}$ correspond to different fan structures (see Figure \ref{fig:fanstruct}).  We label edges between vertices of type ${a_2}$ and ${a_3}$ by ${b_1}$, and the remaining edges by ${b_2}$.  All two-dimensional polytopes are of the same type, which we label by $c$.}
  \label{fig:tetface}
  \end{figure}

  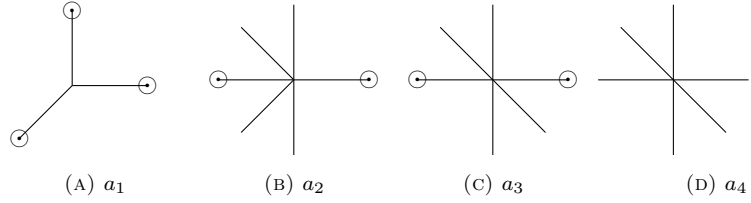
\begin{figure}
        \centering
        \begin{subfigure}[b]{0.2\textwidth}
                \begin{tikzpicture}
    \coordinate (p0) at (0,0) ;
    \coordinate (p1) at (1,0) ;
    \coordinate (p2) at (0,1);
    \coordinate (p3) at (-.7,-.7);
 	\foreach \x in {1,2,3}{
	\draw (p0) -- (p\x);
	\node at (p\x) {$\odot$};

	}
 \end{tikzpicture}
  \caption{${{a_1}}$}
                \label{fig:alpha0}
        \end{subfigure}
                  \begin{subfigure}[b]{0.2\textwidth}    \begin{tikzpicture}
    \coordinate (p0) at (0,0) ;
    \coordinate (p1) at (1,0) ;
    \coordinate (p2) at (0,1);
    \coordinate (p3) at (0,-1);
    \coordinate (p4) at (-1,0);
    \coordinate (p5) at (-.7,-.7);
    \coordinate (p6) at (-.7,.7);

     	\foreach \x in {1,2,...,6}{
	\draw (p0) -- (p\x);
	
	}
\node at (p1) {$\odot$};
\node at (p4) {$\odot$};
	
 \end{tikzpicture}
   \caption{${{a_2}}$}
                \label{fig:alpha2}
        \end{subfigure}
        \begin{subfigure}[b]{0.2\textwidth}
                  \begin{tikzpicture}
    \coordinate (p0) at (0,0) ;
    \coordinate (p1) at (1,0) ;
    \coordinate (p2) at (0,1);
    \coordinate (p3) at (0,-1);
    \coordinate (p4) at (-1,0);
    \coordinate (p5) at (-.7,.7);
    \coordinate (p6) at (.7,-.7);

     	\foreach \x in {1,2,...,6}{
	\draw (p0) -- (p\x);}
	\node at (p1) {$\odot$};
	\node at (p4) {$\odot$};
 \end{tikzpicture}                 \caption{${{a_3}}$}
                \label{fig:bullet}
        \end{subfigure}
                     \begin{subfigure}[b]{0.25\textwidth}
                 \begin{tikzpicture}
    \coordinate (p0) at (0,0) ;
    \coordinate (p1) at (1,0) ;
    \coordinate (p2) at (0,1);
    \coordinate (p3) at (0,-1);
    \coordinate (p4) at (-1,0);
    \coordinate (p5) at (-.7,.7);
    \coordinate (p6) at (.7,-.7);

     	\foreach \x in {1,2,...,6}{
	\draw (p0) -- (p\x);}
 \end{tikzpicture}                 \caption{${{a_4}}$}
                \label{fig:alpha4}
        \end{subfigure}

                \caption{Fan structures.}\label{fig:fanstruct}
\end{figure}
  \begin{figure}
        \centering
        \begin{subfigure}[b]{0.25\textwidth}
                \begin{tikzpicture}
    \coordinate (p0) at (0,0) ;
    \coordinate (p1) at (-1,0) ;
    \coordinate (p2) at (0,-1);
    \coordinate (p3) at (p0);
 	\foreach \x in {0,1,2}{
	\setcounter{countera}{\x}
	\addtocounter{countera}{1}
	\draw (p\x) -- (p\thecountera);
	\node at (p\x) {$\bullet$};

	}
 \end{tikzpicture}
  \caption{${\check{a_1}}$}
                \label{fig:checkalpha1}
        \end{subfigure}
                  \begin{subfigure}[b]{0.25\textwidth}    \begin{tikzpicture}
                     \draw [decorate,decoration={brace,amplitude=3, },xshift=-2pt,yshift=0pt]
(-1,1) -- (-1,2) node [black,midway,xshift=-8pt]
{\footnotesize $\check{b}_2$};
    \coordinate (p0) at (0,0) ;
    \coordinate (p1) at (-1,1) ;
    \coordinate (p2) at (-1,2);
    \coordinate (p3) at (0,3);
    \coordinate (p4) at (1,3);
    \coordinate (p5) at (1,0);
    \coordinate (p6) at (p0);

     	\foreach \x in {0,1,2,...,5}{
	\setcounter{countera}{\x}
	\addtocounter{countera}{1}
	\draw (p\x) -- (p\thecountera);
	\node at (p\x) {$\bullet$};

	}
\node at (0,1) {$\bullet$};
\node at (0,2) {$\bullet$};
\node at (1,1) {$\bullet$};
\node at (1,2) {$\bullet$};

 \end{tikzpicture}
   \caption{${\check{a_2}}$}
                \label{fig:checkalpha2}
        \end{subfigure}
        \begin{subfigure}[b]{0.25\textwidth}
                  \begin{tikzpicture}
   \draw [decorate,decoration={brace,amplitude=3, mirror},xshift=2pt,yshift=0pt]
(1,1) -- (1,4) node [black,midway,xshift=8pt]
{\footnotesize $\check{b}_1$};
    \coordinate (p0) at (0,0) ;
    \coordinate (p1) at (-1,0) ;
    \coordinate (p2) at (-1,3);
    \coordinate (p3) at (0,4);
    \coordinate (p4) at (1,4);
    \coordinate (p5) at (1,1);
    \coordinate (p6) at (p0);

     	\foreach \x in {0,1,2,...,5}{
	\setcounter{countera}{\x}
	\addtocounter{countera}{1}
	\draw (p\x) -- (p\thecountera);
	\node at (p\x) {$\bullet$};

	}
\node at (0,1) {$\bullet$};
\node at (0,2) {$\bullet$};
\node at (0,3) {$\bullet$};
\node at (1,1) {$\bullet$};
\node at (1,2) {$\bullet$};
\node at (1,3) {$\bullet$};
\node at (-1,1) {$\bullet$};
\node at (-1,2) {$\bullet$};
\node at (-1,3) {$\bullet$};
 \end{tikzpicture}                 \caption{$\check{a_3}$}
                \label{fig:checkalpha3}
        \end{subfigure}
\begin{subfigure}[b]{0.2\textwidth}
                 \begin{tikzpicture}
    \coordinate (p0) at (0,0) ;
    \coordinate (p1) at (-1,0) ;
    \coordinate (p2) at (-1,1);
    \coordinate (p3) at (0,2);
    \coordinate (p4) at (1,2);
    \coordinate (p5) at (1,1);
    \coordinate (p6) at (p0);

     	\foreach \x in {0,1,2,...,5}{
	\setcounter{countera}{\x}
	\addtocounter{countera}{1}
	\draw (p\x) -- (p\thecountera);
	\node at (p\x) {$\bullet$};

	}
\node at (0,1) {$\bullet$};
	
 \end{tikzpicture}                 \caption{$\check{a_4}$}
                \label{fig:checkalpha4}
        \end{subfigure}

                \caption{Newton polytopes}\label{fig:newton}
\end{figure}
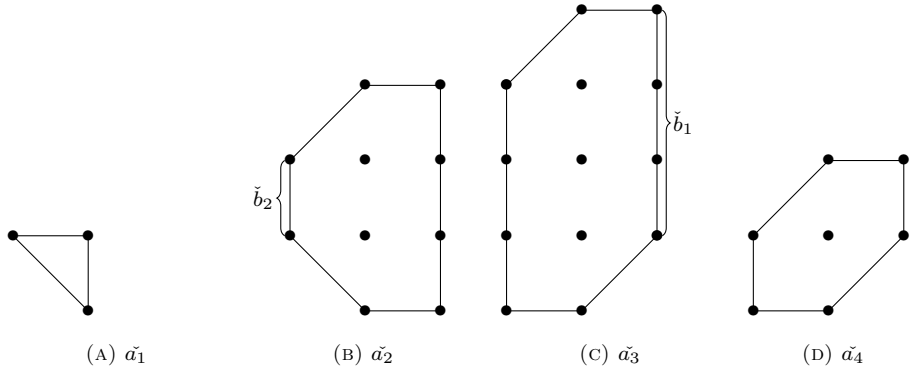
\begin{figure}
                 \resizebox{.5\textwidth}{!}{
                  \begin{tikzpicture}

	\draw (1,5) -- (1,6) -- (2,7) -- (3,7) -- (3,6) -- (2,5) -- cycle;
	\node at (2,6) {$\check{a_4}$};
	\draw (2,4) -- (2,5) -- (3,6) -- (4,6) -- (4,5) -- (3,4) -- cycle;
	\node at (3,5) {$\check{a_4}$};	
	\draw (0,3) -- (0,4) -- (1,5) -- (2,5) -- (2,4) -- (1,3) -- cycle;
	\node at (1,4) {$\check{a_4}$};	
	\draw (0,3) -- (-1, 2);
	\draw (-1, 2) -- (-1,1.5);	
	\draw [dashed] (-1, 1.5) -- (-1,1);	
	\draw (-1, 2) -- (-1.5, 2);	
	\draw [dashed] (-1.5, 2) -- (-2, 2);		
	\draw (0,4) -- (-1.5, 4);
	\draw [dashed] (-1.5, 4) -- (-3,4);	
	\draw (1,6) -- (-0.5, 6);
	\draw [dashed] (-0.5,6) -- (-2, 6);	
	\draw (2,7) -- (2,8);
	\draw (2,8) -- (1.5,8);	
	\draw [dashed] (1.5,8) -- (1,8);		
	\draw (2,8) -- (2.5,8.5);
	\draw [dashed] (2.5,8.5) -- (3,9);	
			
	\draw (3,7) -- (4.5,8.5);
	\draw [dashed] (4.5, 8.5) -- (6,10);
	\draw (4,6) -- (5.5,7.5);
	\draw [dashed] (5.5, 7.5) -- (7,9);	
	\draw (4,5) -- (5,5);
	\draw (5,5) -- (5.5,5.5);
	\draw [dashed] (5.5,5.5) -- (6,6);	
	\draw (5,5) -- (5,4.5);
	\draw [dashed] (5,4.5) -- (5,4);	
	\draw (3,4) -- (3,2.5);
	\draw [dashed] (3,2.5) -- (3, 1);
	\draw (1,3) -- (1,1.5);
	\draw [dashed] (1,1.5) -- (1,0);
	\draw (1,0) -- (2,0)  -- (3,1) -- (4,2) -- (5,4) -- (6,6) -- (7,8) --(7,9) -- (7,10) -- (6,10) -- (5,10) -- (3,9) -- (1,8) -- (-1, 7) -- (-2,6) -- (-3, 5) -- (-3,4) -- (-3, 3) -- (-2,2) -- (-1,1) -- (0,0)-- cycle;
		
	\node at (-1,1.5) {$\odot$};
	\node at ($(-1,1.5)!0.5!(1,1.5)$) {$\check{a_2}$};
	\node at (1,1.5) {$\odot$};
	\node at ($(1,1.5)!0.5!(3,2.5)$) {$\check{a_3}$};	
	\node at (3,2.5) {$\odot$};
	\node at ($(5,4.5)!0.5!(3,2.5)$) {$\check{a_2}$};		
	\node at (5,4.5) {$\odot$};
	\node at ($(5,4.5)!0.5!(5.5, 5.5)$) {$\check{a_1}$};		
	\node at (5.5,5.5) {$\odot$};
	\node at ($(5.5,7.5)!0.5!(5.5, 5.5)$) {$\check{a_2}$};		
	\node at (5.5, 7.5) {$\odot$};
	\node at ($(5.5,7.5)!0.5!(4.5, 8.5)$) {$\check{a_3}$};			
	\node at (4.5, 8.5) {$\odot$};
	\node at ($(2.5,8.5)!0.5!(4.5, 8.5)$) {$\check{a_2}$};		
	\node at (2.5, 8.5) {$\odot$};	
	\node at ($(2.5,8.5)!0.5!(1.5, 8)$) {$\check{a_1}$};		
	\node at (1.5, 8) {$\odot$};	
	\node at ($(-0.5,6)!0.5!(1.5, 8)$) {$\check{a_2}$};			
	\node at (-0.5, 6) {$\odot$};	
	\node at ($(-0.5,6)!0.5!(-1.5, 4)$) {$\check{a_3}$};			
	\node at (-1.5, 4) {$\odot$};	
	\node at ($(-1.5,2)!0.5!(-1.5, 4)$) {$\check{a_2}$};			
	\node at (-1.5, 2) {$\odot$};
	\node at ($(-1.5,2)!0.5!(-1, 1.5)$) {$\check{a_1}$};		
 \end{tikzpicture}}
 \caption{ A chart of the pair $(\check{{B}}, \check{\mathcal{P}})$ dual to the patch of $({B}, \mathcal{P})$ depicted in Figure \ref{fig:tetface}.  The labelings ${\check{a}_i}$ refer to two-dimensional polyhedra dual to the vertices in Figure \ref{fig:tetface}.  Edges of type $\check{b_1}$ separate cells of type $\check{a_2}$ and $\check{a_3}$, while all other edges are of type $\check{b_2}$.  Each vertex is of type $\check{c}$ with identical fan structure given by the fan of $\PP^2$.  The markings $\odot$ again indicate affine singularities, while the dashed segments are ``cut lines" across which the affine structure changes in this chart.}
 \label{fig:dualface}
 \end{figure}
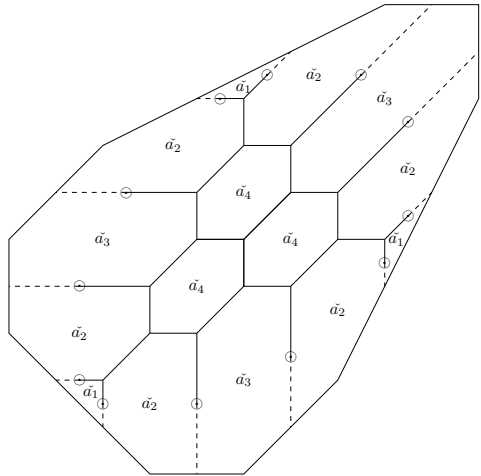
\subsection{Degenerations and log smooth models}
A general description of how to obtain toric degenerations from Batyrev-Borisov Calabi-Yau manifolds was given in \cite{G05}. We follow the slight variant of this given in \cite{RSTZ14}.
\subsubsection{A degeneration by compactifying a pencil}
Let $P$ be the Newton polytope of some hypersurface in $(\CC^*)^n$ that we wish to degenerate.
Assume the origin is the unique interor point of $P$ and that all facets have integral distance one to it, i.e. $P$ is reflexive. 
Let $f$ be a Laurent polynomial defining the hypersurface.
Let $\bar C(P)$ be the sideways truncated cone over $P$, i.e.
$$\bar C(P) = \{(rm,r)\in\RR^{n+1} \mid r\ge 0, m\in P, rm\in P \}.$$
Let $\mathcal{Y}'$ denote the toric variety associated to $\bar C(P)$. We have a proper map $\mathcal{Y}'\ra\A^1$ given by the monomial $z^{e_{n+1}}$ that is associated to the last unit vector $e_{n+1}$ in $\RR^{n+1}$.
We assume that the coefficients of $f$ are generic.
Next consider the Newton polytope $\tilde P$ of the pencil in $(\CC^*)^n$ generated by the original hypersurface and the empty set, i.e.
$$ \tilde P = \operatorname{Newton}( fz^{e_{n+1}}+1) = \operatorname{conv}(P\times\{e_{n+1}\},0).$$
This is the convex hull of the origin and a copy of $P$ placed at height one (with respect to $e_{n+1}$), $z^{e_{n+1}}$ is the pencil parameter.
Since $\tilde P$ coincides with the cut-off of $\bar C(P)$ at $e^*_{n+1}\le 1$, the hypersurface $\mathcal{X}'$ defined by $fz^{e_{n+1}}+1$ doesn't contain any zero-dimensional orbits of $\mathcal{Y}'$.
The restriction of $\mathcal{Y}'\ra \A^1$ to $\mathcal{X}'$ gives a proper map $\pi':\mathcal{X}'\ra \A^1$. This is a degeneration of the compactified hypersurface associated to $f$ (set $z^{e_1}=\infty$ to obtain the original hypersurface).
\subsubsection{Resolving singularities in codimension two}
\label{resolve-codim2}
Note that $\mathcal{Y}'$ and $\mathcal{X}'$ are typically very singular. Let $D'$ denote the horizontal divisor in $\mathcal{Y}'$ given by the vertical facets of $\bar C(P)$ (those that arise from the sideways truncation).
A resolution of singularities in codimension two can be done by choosing another lattice polytope $Q$ whose normal fan is a refinement of $P$'s such that the resulting toric variety $\mathcal{Y}$, i.e. the family $\mathcal{Y}\ra\A^1$, has the properties
\begin{enumerate}
\item the general fibre of $\mathcal{Y}$ is regular,
\item away from $D$, the strict transform of $D'$, the family $\mathcal{Y}\ra \A^1$ is semistable, i.e. the central fibre $\mathcal{Y}_0$ is a normal crossing divisor in $Y$.
\end{enumerate}
We will give such degenerations explicitly in the examples that interest us.
We then restrict the family to the strict transform of $\mathcal{X}'$ to obtain a degenerating family $\pi:\mathcal{X}''\ra \A^1$
for the quartic and its mirror dual.

\subsubsection{Resolving the remaining log singularities}
We now assume $\dim \mathcal{X}''=3$, i.e. $\mathcal{X}''$ is a degeneration of surfaces.
Denote by $\mathcal{X}_0=\pi^{-1}(0)$ the special fibre.
It is not convenient to deal with the remaining log singularities in $\mathcal{X}''$ torically.
If there are some, these lie in the intersection of $D$ and $\operatorname{Sing}\mathcal{X}_0$, hence they are of codimension three in $\mathcal{X}''$.
Note that $\mathcal{Y}$ is smooth at the generic points of $D\cap \mathcal{X}_0$ by the reflexivity assumption.
In the quartic and its mirror, the set of points where $\mathcal{X}''\ra\A^1$ is not log smooth is $D\cap \operatorname{Sing}\mathcal{X}_0$
and each singularity is an ordinary double point. At such a singularity, $\mathcal{X}''\ra\A^1$ takes the shape
$$xy-tw=0$$
where $t$ is the coordinate of $\A^1$. We resolve all of these simultaneously by a projective small resolution $\mathcal{X}\ra \mathcal{X}''$ simply by blowing up components of $\mathcal{X}_0$ in $\mathcal{X}''$ until all singularities are gone. Each singularity in the resolution is replaced by a $\mathbb{P}^1$ that is contained in the strict(=proper) transform of the component of $\mathcal{X}_0$ that was blown up to resolve this singularity.
Let $\tilde{\mathcal{X}}_0$ be the strict transform of $\mathcal{X}_0$ in $\mathcal{X}$.
What is most important for us is that motivically
\begin{equation}
\label{additional-lines}
[\tilde{\mathcal{X}}_0]=[\mathcal{X}_0]+m\LL
\end{equation}
where $m$ is the number of singularities. We will also need that the copies of $\LL$ don't lie in any proper intersection of two strata of $\tilde{\mathcal{X}}_0$. In other words, they are added to the part of a component of $\mathcal{X}_0$ that is the complement of the proper intersections with other components.
We will have $m=24$ for the quartic and its mirror-dual.
We denote the resulting semistable family by $\pi:\mathcal{X}\ra \A^1$.

\subsubsection{Degeneration of the quartic}
The Newton polytope $P$ of the quartic was given in \eqref{quarticpoly}. We construct the family $\mathcal{X}'\ra\A^1$ as in the previous subsections. The general fibre is already smooth. It remains to make the central fibre semistable. For this, we choose a piecewise linear function $\hat\varphi: P\ra\RR$ that is uniquely determined by its values at the lattice points in the boundary of $P$ which we choose on each facet to be as given in Figure~\ref{fig:tetface}. We then define the polytope $Q$ to be the convex hull of the graph of $\hat\varphi$, i.e.
$$Q = \{(m,r)\in\RR^{n+1} \mid r\ge \hat\varphi(m)\}.$$
We leave it to the interested reader to check that
\begin{enumerate}
\item The restriction of $\hat\varphi$ to each facet induces the subdivision given in Figure~\ref{fig:tetface},
\item Let $\tau$ be a simplex in the subdivision of a facet then the piecewise linear function induced by $\hat\varphi$ on the quotient fan of $P$ by $\tau$ is up to addition of a linear function given by the Newton polytope $\check\tau$ in Fig.~\ref{fig:newton}.
\end{enumerate}
Using $Q$, we obtain $\mathcal{X}\ra \A^1$ as in the previous subsections.
Since the quartic meets the singular locus of the union of hyperplanes in $\PP^3$ in $24$ points, we find $m=24$.

\subsubsection{Degeneration of the quartic mirror dual}
To obtain the mirror dual degeneration $\check{\mathcal{X}}\ra\A^1$ we proceed analogously. Let $\check P$ denote the reflexive dual of $P$.
The Batyrev quartic mirror dual is the crepant resolution of a general hypersurface in the projective toric variety associated to $\check P$.
Using $\check P$ we obtain $\bar C(\check P)$, $\check{\mathcal{Y}}\ra\A^1$ and $\check {\mathcal{X}}'\ra\A^1$ similar to before.
Let $\check Q_1$ denote the Newton polytope of the piecewise linear function $\hat\varphi$ that we constructed for the resolution of the quartic degeneration.
Let $\hat\varphi_0:P\ra\RR$ be the piecewise linear function that takes value $\hat\varphi(v)-1$ at a lattice point $v$ in $\partial P$ and let $\check Q_0$ be its Newton polytope. This results in a Minkowski sum.
$$\check Q_1=\check Q_0+\check P.$$
Finally, we set
$$\check Q = \{ (q_0,0)+r(p,1)\in\RR^{n+1} \mid q_0\in \check Q_0, p\in \check P, r\ge 0\}$$
We leave it to the interested reader to check that this satisfies the two items in \S\ref{resolve-codim2}.
Again there are $m=24$ double points to resolve to obtain a semistable family $\check{\mathcal{X}}\ra\A^1$ and the intersection complex of the central fibre is the one depicted in Figure~\ref{fig:dualface}.

\subsection{Applying Bultot's theorem}
With the aid of Bultot's formula for log smooth models (Theorem \ref{thm:bultot}), we use the combinatorial data of $(B, \mathcal{P})$ and $(\check{B}, \check{\mathcal{P}})$ to compute the motivic zeta function of the quartic, the general fiber of $\mathcal{X}$.  The same argument can be applied (after applying the appropriate dualization) to the mirror quartic.

  The polyhedral complex $\mathcal{P}$ is the intersection complex of $\mathcal{X}_0$, encoding its toric strata, while $\check{\mathcal{P}}$ encodes the discrete part of its log structures.  When translated into the language of our construction, incorporating \eqref{additional-lines}, Theorem \ref{thm:bultot} yields
\begin{align}
Z_{\textrm{quartic}}(T)=&24 \mathbb{L}\frac{T}{1-T}+\sum_{\tau\in \mathcal{P}} (\mathbb{L}-1)^{r(\tau)} [U(\tau)]\sum_{\phi \in \overline{M}^{\vee,\mathrm{loc}}_{\mathcal{X},\tau}}T^{\phi(e_{\xi})}\\
=&24\mathbb{L}\frac{T}{1-T}+\sum_{\tau\in \mathcal{P}} (\mathbb{L}-1)^{r(\tau)} [U(\tau)]\sum_{(\mathbf{m}, t)\in C(\check{\tau})^\circ\cap \mathbb{Z}^{r(\tau)}\times \mathbb{N}}T^{n}\label{gszeta}.
\end{align}

Here we've identified $\check{\tau}$ with a polyhedron in $\mathbb{Z}^{r(\tau)}$, $C(\check{\tau})^\circ\subseteq \mathbb{R}^{r(\tau)}\times\mathbb{R}_{\geq 0}$ is the interior of the cone over $\check{\tau}$, $r(\tau)$ is the codimension of $\tau$ in $B$, and
$$U(\tau):=\mathbb{P}_\tau\setminus \bigcup_{\substack{\sigma\in \mathcal{P} \\ \sigma\subsetneq \tau}} \mathbb{P}_\sigma.$$  
The equality above follows from the identification of  $\overline{M}_{\mathcal{X},\tau}$ with the set of points $(\mathbf{n}, t)\in \mathbb{Z}^{r(\tau)}\times \mathbb{N}$ with $t\geq \varphi_\tau (\mathbf{n})$, where $\varphi_\tau$ is a representative of $\varphi$ on the \emph{fan structure} along $\tau$ (see \cite{GrSi}).  Then, by definition of $\check{\tau}$, $C(\check{\tau})\cap \mathbb{Z}^{r(\tau)}\times \mathbb{N}$ can be identified with $\overline{M}^{\vee}_{\mathcal{X},\tau}$, with $C(\check{\tau})^\circ\cap \mathbb{Z}^{r(\tau)}\times \mathbb{N}=\overline{M}^{\vee,\mathrm{loc}}_{\mathcal{X},\tau}$.  If $(\mathbf{m}, t)$ corresponds to $\phi \in \overline{M}^{\vee,\mathrm{loc}}_{\mathcal{X},\tau}$, then $t=\phi(e_{\tau})$.
Then
\begin{align*}
Z_{\textrm{quartic}}(T)=24 \mathbb{L}\frac{T}{1-T}+\sum_{\tau\in \mathcal{P}} (\mathbb{L}-1)^{r(\tau)} [U(\tau)]\sum_{n\geq 1}l_{\check{\tau}}(n)T^{n}
\end{align*}
where $$l_{\check{\tau}}(n):=\#\left\{(\mathbf{m}, t)\in C(\check{\tau})^\circ\cap \mathbb{Z}^{r(\tau)}\times \mathbb{N}|t=n\right\}.$$
As noted in \cite[Proof of Proposition 2.2]{GKR12}, one can compute the function $l_{\check{\tau}}$ by counting simplices in a unimodular triangulation of $\check{\tau}$.
Let $\P_{\check{\tau}}$ be the set of simplices of such a subdivision (which always exists for $\dim\check{\tau}\le 2$).
The number of integral lattice points in the interior of the $j$th scaling of the standard $k$-simplex is given by the coefficient of $T^j$ in the expansion about $0$ of $\left(\frac{T}{1-T}\right)^{k+1}$.  We obtain
\begin{align*}
\sum_{n\geq1} l_{\check{\tau}}(n) T^n= \sum_{\substack{ \sigma \in  \P_{\check{\tau}} \\ \sigma\not \subseteq \partial\check{\tau} }}\left(\frac{T}{1-T}\right)^{{\rm dim}(\sigma)+1}.
\end{align*}
Combining this result with the observation that $U(\tau)=(\mathbb{L}-1)^{2-r(\tau)}$, we can rewrite \ref{gszeta}.  Consider $\tilde{\check{\mathcal{P}}}$, a unimodular triangulation of $\check{\mathcal{P}}$
As before, let $t(\textrm{quartic})$ be the number of triangles appearing in $\tilde{\check{\mathcal{P}}}$.  Then
\begin{align*}
Z_{\textrm{quartic}}(T)=&24 \mathbb{L}\frac{T}{1-T}+\sum_{\tau\in \mathcal{P}} (\mathbb{L}-1)^{2}\sum_{\substack{\sigma\in\tilde{\check{\mathcal{P}}}\\  \sigma \subseteq \check{\tau} \\ \sigma\not \subseteq \partial\check{\tau} }}\left(\frac{T}{1-T}\right)^{3-r(\sigma)} \\
=&24 \mathbb{L}\frac{T}{1-T}+(\mathbb{L}-1)^{2}\sum_{\sigma\in\tilde{\check{\mathcal{P}}}}\left(\frac{T}{1-T}\right)^{3-r(\sigma)} \\
=&24 \mathbb{L}\frac{T}{1-T}+(\mathbb{L}-1)^2\Bigg[t(\textrm{quartic})\left(\frac{T}{1-T}\right)^3\\&+\frac{3t(\textrm{quartic})}{2}\left(\frac{T}{1-T}\right)^2+\left(\frac{t(\textrm{quartic})}{2}+2\right)\left(\frac{T}{1-T}\right)\Bigg]\\
=&\frac{t(\textrm{quartic})}{2}(\mathbb{L}-1)^2\frac{T(1+T)}{(1-T)^3}+(1+10\mathbb{L}+\mathbb{L}^2)\frac{2T}{1-T}.
\end{align*}
 Thus we see agreement with Theorem \ref{thm-StVo}.

Each facet of the tetrahedron $P$ can be triangulated as in Fig.~\ref{fig:tetface}. This yields $4\cdot 16$ triangles giving $t(\textrm{mirror quartic})=64$.  Letting $t(\tau)$ denote the count of lattice triangles in an elementary triangulation of a polygon $\tau$, using the labelling as in Fig.~\ref{fig:newton}, we obtain $t(\textrm{quartic})$ as
$$4t(\check{a_1})+2\cdot 6t(\check{a_2})+ 6t(\check{a_3}) +3\cdot 4t(\check{a_4}) = 4 + 2\cdot 6\cdot 6 + 6\cdot 14 + 3\cdot 4\cdot 10=   280.$$
Thus,
\begin{align*}
Z_{\textrm{quartic}}(T)=&140(\mathbb{L}-1)^2\frac{T(1+T)}{(1-T)^3}+(1+10\mathbb{L}+\mathbb{L}^2)\frac{2T}{1-T} \\
Z_{\textrm{mirror quartic}}(T)=&32(\mathbb{L}-1)^2\frac{T(1+T)}{(1-T)^3}+(1+10\mathbb{L}+\mathbb{L}^2)\frac{2T}{1-T}.
\end{align*}

\section{Arithmetic of the monodromy}
Let us assume now that $\mathcal{X}$ is a type III degeneration of K3 surfaces, $\mathcal{X}_s$ the nearby fibre and $\mathcal{X}_0$ the central fibre.
We follow \cite{FS86}.
Let $\mathcal{T}\in \Aut(H^2(\mathcal{X}_s; \ZZ))$ be the Picard Lefschetz transformation (i.e. monodromy).
One sets
$$N=\log \mathcal{T}=(\mathcal{T}-1)-\frac12(\mathcal{T}-1)^2.$$
Denote by $$W_0\subset W_2\subset W_4=H^2(\mathcal{X}_s; \QQ)$$
the weight filtration induced by $N$, i.e. $W_0=\mathrm{im}\, N^2$, $W_2=\ker N^2=W_0^\perp$.
We have that $W_0\cap H^2(\mathcal{X}_s; \ZZ)\cong \ZZ$, let $\gamma$ be a generator.
Let $(x\cdot y)$ denote the pairing of $x,y\in H^2(X_s; \QQ)$.
By unimodularity, there is a $\gamma'\in H^2(\mathcal{X}_s; \ZZ)$ such that $(\gamma\cdot\gamma')=1$.
Set $$\delta=N\gamma'\in W_2.$$

\begin{proposition}[Friedman and Scattone]
We have for $x\in H^2(\mathcal{X}_s; \QQ)$,
$$ Nx = (x\cdot \gamma)\delta - (x\cdot\delta)\gamma$$
and $\delta\in H^2(\mathcal{X}_s; \ZZ)$, so $N$ preserves $H^2(\mathcal{X}_s; \ZZ)$.
\end{proposition}
\begin{proof}\cite[Lemma 1.1]{FS86}.\end{proof}

There is a pair of integers $(t,k)$ that determines $N$ up to conjugation \cite[Prop. 1.7]{FS86}.
One defines
$$t\stackrel{\textrm{\tiny def}}{=}(\delta\cdot\delta)=\#(\textrm{triangles in a unimodular triangulation of }(\check B,\check \P))$$
where the second equality is \cite[Prop 7.1]{FS86} where it appears as the count of triple points of a semistable birational model for $\mathcal{X}_0$.
The index $k$ is the maximal integer such that $N/k$ is integral, i.e. by the proposition this is the maximal integer such that $\delta/k$ is integral.
Note that $k^2$ divides $t$.

Let $f:\mathcal{X}_s\ra B$ denote the continuous map that is the SYZ fibration \cite{G01}.
It follows from \cite{G01} that $\gamma$ is a fibre and $\gamma'$ is a section.
\begin{lemma}
We have that $\delta=N\gamma'$ is given by the tropical 1-cycle that is the 1-skeleton of $\P$ supporting sections of $\check\Lambda$ that are given by the kinks of $\varphi$.
\end{lemma}
\begin{proof} We leave this to the upcoming paper \cite{GSup} 
and refer to \cite{RSTZ14} that explains how one arrives at this.
\end{proof}

This Lemma allows us to compute the index $k$, i.e. the maximal integer such that $\delta/k$ is integral.
For the computation of $t$, it is a nice exercise to compute $(\delta\cdot\delta)$ using the tropical intersection product but for us it is easier to go by the triangle count.

\begin{theorem} We have
$$
\begin{array}{cc}
t(\textrm{quartic})=280=2^3\cdot5\cdot7, & k(\textrm{quartic})=1,\\
t(\textrm{quartic-mirror})=64=2^6, & k(\textrm{quartic-mirror})=4.\\
\end{array}
$$
\end{theorem}

\begin{figure}
                 \resizebox{.5\textwidth}{!}{
                  \begin{tikzpicture}

	\draw[thick] (2,5) -- node[below right] {4} (7,10);
	\draw[thick] (2,5)  -- node[right] {4} (2,0);
	\draw[thick] (2,5)  -- node[above] {4} (-3,5);

	\draw [dashed] (-1, 1.5) -- (-1,1);	
	
	\draw [dashed] (-1.5, 2) -- (-2, 2);		

	\draw [dashed] (-1.5, 4) -- (-3,4);	

	\draw [dashed] (-0.5,6) -- (-2, 6);	

	\draw [dashed] (1.5,8) -- (1,8);		

	\draw [dashed] (2.5,8.5) -- (3,9);	

	\draw [dashed] (4.5, 8.5) -- (6,10);

	\draw [dashed] (5.5, 7.5) -- (7,9);	

	\draw [dashed] (5.5,5.5) -- (6,6);	

	\draw [dashed] (5,4.5) -- (5,4);	

	\draw [dashed] (3,2.5) -- (3, 1);

	\draw [dashed] (1,1.5) -- (1,0);
	\draw (1,0) -- (2,0)  -- (3,1) -- (4,2) -- (5,4) -- (6,6) -- (7,8) --(7,9) -- (7,10) -- (6,10) -- (5,10) -- (3,9) -- (1,8) -- (-1, 7) -- (-2,6) -- (-3, 5) -- (-3,4) -- (-3, 3) -- (-2,2) -- (-1,1) -- (0,0)-- cycle;
		
	\node at (-1,1.5) {$\odot$};
	\node at (1,1.5) {$\odot$};
	\node at (3,2.5) {$\odot$};
	\node at (5,4.5) {$\odot$};
	\node at (5.5,5.5) {$\odot$};
	\node at (5.5, 7.5) {$\odot$};
	\node at (4.5, 8.5) {$\odot$};
	\node at (2.5, 8.5) {$\odot$};	
	\node at (1.5, 8) {$\odot$};	
	\node at (-0.5, 6) {$\odot$};	
	\node at (-1.5, 4) {$\odot$};	
	\node at (-1.5, 2) {$\odot$};
 \end{tikzpicture}}
 \caption{ The tropical curve given by the $1$-skeleton in Fig.~\ref{fig:dualface} is homologous to the tropical curve depicted here}
 \label{fig:fourtimesprimitive}
 \end{figure}
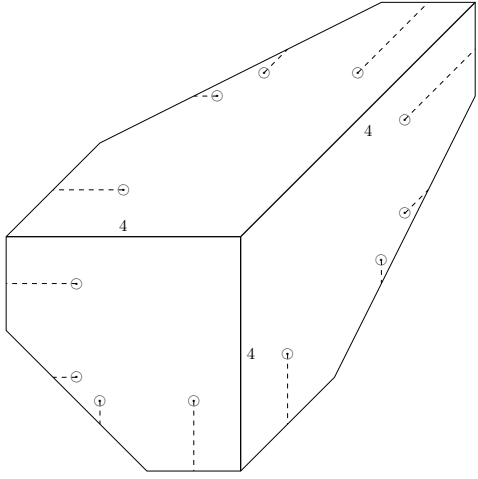
  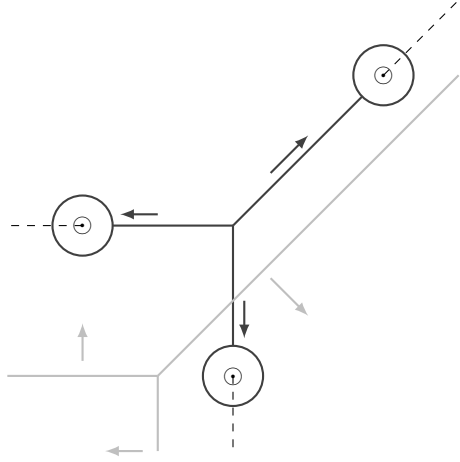
\begin{figure}
                  \begin{tikzpicture}

	\draw [thick, color=darkgray](2,2) -- (0,0);
	\draw [thick, color=darkgray,   -latex] (.5,.7) -- (1,1.2);	
	\draw [thick, color=darkgray](-2,0) -- (0,0);
	\draw [thick, color=darkgray,   -latex] (-1,0.15) -- (-1.5,0.15);	
	\draw [thick, color=darkgray](0,-2) -- (0,0);		
	\draw [thick, color=darkgray,   -latex] (0.15,-1) -- (0.15,-1.5);

	\draw[thick, color=lightgray] (3,2) -- (-1,-2);
	\draw [thick, color=lightgray,   -latex] (0.5, -0.7) -- (1,-1.2);	
	\draw [thick, color=lightgray](-3,-2) -- (-1,-2);
	\draw [thick, color=lightgray,   -latex] (-2,-1.8) -- (-2,-1.3);	
	\draw [thick, color=lightgray](-1,-3) -- (-1,-2);	
	\draw [thick, color=lightgray,   -latex] (-1.2,-3) -- (-1.7,-3);		
	\draw (0,-2) [color=darkgray, thick, fill=white] circle  (.4cm);
	\draw (-2,0) [color=darkgray, thick, fill=white] circle  (.4cm);
	\draw (2,2) [color=darkgray, thick, fill=white] circle  (.4cm);
	\node at (2,2) {$\odot$};
	\node at (-2,0) {$\odot$};
	\node at (0,-2) {$\odot$};	
	\draw[dashed] (2,2) -- (3,3);
	\draw[dashed] (0,-2) -- (0,-3);	
	\draw[dashed] (-2,0) -- (-3,0);

 \end{tikzpicture}
 \caption{Simple intersection of a tropical cycle and cocycle }
 \label{fig:intersection-of-cycles}
 \end{figure}

\begin{proof}
The quantities $t$ were computed in the previous section.  Consider Fig.~\ref{fig:dualface}. The tropical curve $\alpha$ obtained from the 1-skeleton in the intersection complex of the quartic-mirror is the tropical curve part of which is shown in Fig.~\ref{fig:dualface} and the sections of $\check\Lambda$ are primitive along each edge. Furthermore, the sections are invariant under local monodromy wherever an edge passes through a singularity. This allows us to deform the edges over the singularities by adding a suitable tropical 2-cycle. This way, it can be seen that $\alpha$ is homologous to the curve depicted in Fig.~\ref{fig:fourtimesprimitive}.
Let $\alpha'$ be the reduced tropical curve, i.e. $\alpha=4\alpha'$. We need to show that $k=4$ is the maximum, i.e. that $\alpha'$ is not a homologous to a proper multiple of another curve.
This follows from looking at $t=t(\textrm{quartic})=64$. Since $k^2$ divides $t$, the only other options would be $k=8$ but then, by \cite[Fig. B]{FS86}, $\P$ would need to be the refinement of a triangulation of the sphere with a single triangle which is impossible, hence $k=4$.
It remains to show that $k(\textrm{quartic})=1$. It suffices to exhibit a cocycle that pairs to $1$ with the tropical cycle in given by the 1-skeleton in the intersection complex of the quartic, see Fig.~\ref{fig:tetface}. One such cycle is shown in black in Fig~\ref{fig:intersection-of-cycles}. It is Y-shaped where each end encircles a singularity that neighbours a fixed corner of the tetrahedron.
Let's denote this one by $\alpha$. We perturb the 1-skeleton cycle slightly so that it intersects $\alpha$ transversely in a single point, let's call this $\beta$. The pairing of the sections of $\Lambda$ and $\check\Lambda$ at the stalk of the intersection point is $\pm 1$, see Fig~\ref{fig:intersection-of-cycles}.
This pairing coincides with the homology-cohomology pairing of the usual (co-)cycles in $\hat\alpha\in H^1(\mathcal{X}_s,\ZZ)$ and $\hat\beta\in H_1(\mathcal{X}_s,\ZZ)$ reconstructed from $\alpha$ and $\beta$ respectively as in \cite{RS14}.
This implies that $\hat\beta$ is primitive and thus $k(\textrm{quartic})=1$.
\end{proof}
Let us consider the Lefschetz operator $L\in\operatorname{End}(H^\bullet(\mathcal{X}_s,\ZZ))$ on the cohomology of the quartic, i.e. $L$ is the cup product with $4H$ for $H$ the hyperplane class.
We may define $\check k(\hbox{quartic})$, $\check t(\hbox{quartic})$ using $L$ analogous to how we defined $k,t$ using $N$.
This means $\check k(\hbox{quartic})$ is the maximum such that $H/\check k$ is integral, i.e. $\check k(\hbox{quartic})=4$.
We find that $\check t(\hbox{quartic}) := |\operatorname{coker} L^2| = 64$ as this can be computed from the ambient $\PP^3$ where it is $(4H)^2.\hbox{(quartic)}=(4H)^3=64H^3$.

We deduce the following mirror symmetry result.
\begin{corollary}
$\check t(\hbox{quartic})= t(\hbox{quartic-mirror})$,\\
$\check k(\hbox{quartic})= k(\hbox{quartic-mirror})$.
\end{corollary}

\bibliography{quarticbib}

 \end{document}